\RequirePackage[T1]{fontenc}
\documentclass[twoside,12pt]{amsart}

\usepackage{amsfonts}
\usepackage{amsmath}
\usepackage{amsthm}
\usepackage{amssymb}
\usepackage{bbm}
\usepackage{fancyhdr}
\usepackage{tikz}
\usetikzlibrary{positioning,arrows.meta}
\usepackage{etoolbox}
\usepackage{stmaryrd}
\usepackage[all]{xy}
\usepackage{physics}
\usepackage{tikz-cd}
\usepackage{mathrsfs}
\usepackage{xcolor}
\usepackage{latexsym}
\usepackage[normalem]{ulem}
\usepackage{thm-restate}
\usepackage{bussproofs}

\tikzcdset{row sep/normal=50pt, column sep/normal=50pt}

\usepackage[bookmarks=true,bookmarksdepth=3,bookmarksnumbered=true,colorlinks=false,pdfborder={0 0 0}]{hyperref}

\makeatletter
\newcommand{\SafeTocLink}[3] {
\hyperlink{#1}{#2\hfill #3}
\patchcmd{\l@section}
{\@dottedtocline{1}{0em}{1.5em}{#1}{#2}}{\@dottedtocline{1}{0em}{1.5em}{\SafeTocLink[\@currentHref]}
{#1}{#2}}{} }{}{}
\makeatother

\theoremstyle{definition}
\newtheorem{theorem}{Theorem}[section]

\newtheorem{lemma}[theorem]{Lemma}
\newtheorem{corollary}[theorem]{Corollary}
\newtheorem{proposition}[theorem]{Proposition}
\newtheorem{definition}[theorem]{Definition}

\newtheorem{remark}[theorem]{Remark}

\newcommand{\el}{\mathstrut \mathrm{el} }

\newcommand{\id}{\mathstrut \mathrm{id} }

\newcommand{\dom}{\mathstrut \boldsymbol{\mathrm{dom}} }
\newcommand{\cod}{\mathstrut \boldsymbol{\mathrm{cod}} }

\newcommand{\set}{\mathstrut \boldsymbol{\mathrm{set}} }

\newcommand{\mon}{\mathstrut \boldsymbol{\mathrm{mon}} }
\newcommand{\Pair}{\mathstrut \boldsymbol{\mathrm{Pair}} }
\newcommand{\Had}{\mathstrut \boldsymbol{\mathrm{Had}} }

\newcommand{\coloreduline}[3][red]{
\tikz[baseline=(X.bese)]{\node[innee sep=0pt, outer sep=0pt](X) {\textcolor{#2}{#3}};\draw[#1, line width=0.6pt](X.south west) -- (X.south east); } }

\newcommand{\Thmref}[1]{\hyperref[#1] {\coloreduline[red]{black}{\ref*{#1}}}}

\title{\textbf{CROSSED BURNSIDE RINGS FOR GROUPOIDS}}

\author{KEITARO SHIIZUKA}

\thanks{Department of Mathematics, Kindai University, 3-4-1, Kowakae, Higashi-Osaka, Osaka 577-8502, Japan, Email: keitaro.shizuka.math@gmail.com}

\begin{document}

\begin{abstract}
In this paper, we extend the classical theory of crossed $G$-sets and the crossed Burnside ring from a finite group $G$ to a finite groupoid $\mathcal{G}$. We introduce a natural monoidal structure on the category of crossed $\mathcal{G}$-sets over a $\mathcal{G}$-monoid and construct the corresponding crossed Burnside ring of a $\mathcal{G}$-monoid. Finally, we prove a decomposition theorem that expresses the crossed Burnside ring of a groupoid as a product of crossed Burnside rings of groups. 
\end{abstract}

\maketitle

\textbf{Key words and phrases} : Groupoid, action groupoid, groupoid action, crossed $G$-set, crossed Burnside ring

\tableofcontents

\section{Introduction}
The notion of a groupoid was introduced by H. Brandt in 1926 (see \cite{Bra26}). A groupoid is an algebraic structure that generalizes the concept of a group, but it differs essentially from groups in that not every pair of elements can necessarily be composed. Fundamental notions and theorems in group theory, such as Cayley's theorem, Zassenhaus's theorem have their counterparts in the context of groupoid (see \cite{Iva02}, and \cite{AMP20}).
Several equivalent definitions of a groupoid and its actions are known. In this paper, adopting the categorical viewpoint, we regard a groupoid $\mathcal{G}$ as category whose morphisms are all invertible. In this setting, its action is described by a functor $\mathcal{G}\to \set$.
The crossed Burnside ring was introduced in \cite{Yos97}.
There are several papers on crossed Burnside rings, namely \cite{Yos97}, \cite{OY01}, \cite{Bou03} and \cite{Rog13}. \cite{OY01} provides a detailed discussion of crossed $G$-sets and crossed Burnside rings, including explicit computations for the groups $S_3$, $A_4$, $A_5$.
In \cite{Bou03}, a formula for the decomposition of idempotents in positive characteristic is given. In \cite{Rog13}, the relationship between the crossed Burnside ring and the Mackey algebra is examined.  
In \cite{BD20}, it is shown that for a Mackey 2-functor $M$, there is a one-to-one correspondence between the decomposition of the additive category $M(G)$ (The image of a finite group $G$) and the decomposition of the crossed Burnside ring $B^c(G)$. In recent years, Burnside theory for groupoids has been introduced in \cite{ES23}. However, extensions of the crossed Burnside theory within the framework of groupoids do not appear to have been studied yet. Therefore, the purpose of this manuscript is to develop such an extension of crossed Burnside theory to groupoids. 
In Section $2$, we provide basic definitions and notations. In Section $3$, we define a monoidal structure on the category of crossed $\mathcal{G}$-sets over a fixed $\mathcal{G}$-monoid. Then, we define the crossed Burnside ring with respect to the product induced by this monoidal structure. Next, we compare the monoidal structure of $\set^{\mathcal{G}}$ and the category of crossed $\mathcal{G}$-sets over a fixed $\mathcal{G}$-monoid, and show that the Burnside ring of a groupoid $\mathcal{G}$ namely, the Grothendieck ring of the
monoidal category $\set^{\mathcal{G}}$ can be embedded into the crossed Burnside ring of a groupoid $\mathcal{G}$. Finally, we prove the following Theorem, which states that the crossed Burnside ring of a groupoid can be decomposed into a direct product of crossed Burnside rings of groups.
\setcounter{section}{3}
\setcounter{theorem}{7}
\begin{theorem}
If $\mathcal{G}\cong\coprod^n_{i=1}\mathcal{G}_{[i]}$ is the decomposition of $\mathcal{G}$ into its connected components, then two rings $B^c(\mathcal{G})$ and $\prod_{x\in\pi_0(\mathcal{G})}B^c(\mathcal{G}_x)$ are isomorphic.
\end{theorem}
This theorem enables us to compute the crossed  Burnside ring of a groupoid using methods developed for the crossed Burnside ring of a group.
\setcounter{section}{1}

\section{Preliminaries}
\subsection{Notations and conventions}
\begin{itemize}
\item Throughout this paper, calligraphic symbols (e.g., $\mathcal{G}$) denote a finite groupoid. Whereas plain symbols (e.g., $G$), refer to either an object of a groupoid or a one-object groupoid (that is, a group).
\item We identify a finite group $G$ with the one-object groupoid associated with $G$. 
\item The empty category is a finite groupoid; it is excluded from the class of finite groupoids considered in this paper.
\item In a groupoid $\mathcal{G}$, the composition $g\circ g'$ of morphisms is denoted by $gg'$. (This omission is convenient and visually clear when considering groupoids as extensions of groups.) \item In the disjoint union $\coprod_{i\in I} A_i$, the element corresponding to $x\in A_i$ is denoted by $\langle i,x \rangle$. 
\item For a monoid $M$, we denote its unit element by $1_M$.
\item  For a ring $A$, we denote its multiplicative unit element by $1_A$.
\item We use the symbol $\circ$ to denote horizontal composition, and the symbol $*$ to denote vertical composition of natural transformations. We also denote the identity natural transformation on a functor $F$ simply by $F$.
\item Fix a Grothendieck universe $\mathfrak{U}$, and assume that every category $C$ considered in this manuscript is $\mathfrak{U}$-small. That is, the collection of objects $C_0$ and the collection of morphisms $C_1$ both belong to $\mathfrak{U}$. It is well known that the existence of a Grothendieck universe is equivalent to the existence of an inaccessible cardinal; for details on this relationship, see \cite{Wil69}. For a functor $F:C\to D$, we denote by $F_0$ the map on objects and by $F_1$ the map on morphisms.
\item For a category $\mathcal{C}$, define the maps $\dom$ and $\cod$ as follows:
\[\dom:\mathcal{C}_1\to\mathcal{C}_0\;;\;f\mapsto \dom f,\]
\[\cod:\mathcal{C}_1\to\mathcal{C}_0\;;\;f\mapsto \cod f.\]
\item We let $\cong$ denote an isomorphism of categories and $\simeq$ an equivalence of categories.
\end{itemize}
Here we list the notation that will be used frequently without further comment.\\

\begin{tabular}{@{}l l@{}}
 $C(x,y)$& Set of morphisms from $x$ to $y$ in a category $C$\\
 $C/x$&Slice category of $C$ over $x\in C_0$\\
 $\set$& Category of finite sets\\
 $1$& Terminal object of a category\\
 $\mon$& Category of finite monoids\\
$|X|$& Cardinality of a set $X$\\
\end{tabular}

\subsection{Groupoids}
In this subsection, we prepare the groundwork for the subsequent sections. Throughout this paper, we treat groupoids within categorical framework. Algebraic approaches to the topics discussed in this section can be found in \cite{Iva02}, \cite{ES23}, \cite{AMP20}.

\begin{itemize}
\item
A category $\mathcal{G}$ is a \textbf{groupoid} if every morphism in $\mathcal{G}$ is invertible.
\item
$(2)$ For each object $x\in\mathcal{G}_0$, the set $\mathcal{G}(x,x)$, denoted by $\mathcal{G}_x$, is called the \textbf{isotropy group} at $x\in\mathcal{G}_0$.
\item A subcategory $\mathcal{H}$ of a groupoid $\mathcal{G}$  that is itself a groupoid is called a \textbf{subgroupoid} of $\mathcal{G}$.

\item If $\mathcal{N}$ is a wide subgroupoid of $\mathcal{G}$, then $\mathcal{N}$ is said to be a \textbf{normal subgroupoid} if $^g\mathcal{N}_{\dom g}=\mathcal{N}_{\cod g}$ for every $g\in \mathcal{G}_1$. Here, $^g\mathcal{N}_{\dom g}$ denotes the set $\{gng^{-1}\mid n\in \mathcal{N}_{\dom g}\}$.
\end{itemize}

\begin{definition}
If $X$ is a finite set, then groupoid $\Pair(X)$ is the category defined by the following conditions:
\begin{enumerate}
\item $\Pair(X)_0:=X,$
\item $\Pair(X)(x,y):=\{(x,y)\}$, for \;$(x,y)\in X\times X,$
\item The composition of $f=(x,y)$ and $g=(y,z)$ is defined by $gf:=(x,z)$.
\end{enumerate}
The groupoid $\Pair(X)$ is called the \textbf{groupoid of pairs}.
\end{definition}

\begin{itemize}
\item Let $\mathcal{G}$ be a groupoid. For $x,y\in \mathcal{G}_0$, define $x\sim y$ if there exists a morphism $g:x\to y$. Then $\sim$ is an equivalence relation on $\mathcal{G}_0$. The elements of the quotient set $\mathcal{G}_0/{\sim}$ are called the \textbf{connected components} of $\mathcal{G}$, and we denote this set by $\pi_0(\mathcal{G})$. In particular, If $|\pi_0(\mathcal{G})|=1$, then $\mathcal{G}$ is called a \textbf{connected groupoid}.
\end{itemize}

The following proposition is well known.
\begin{proposition}\label{gpdst}
Let $\mathcal{G}$ be a connected groupoid and let $x\in\mathcal{G}_0$. Then, $\mathcal{G}\cong\mathcal{G}_x\times\;\Pair(\mathcal{G}_0)$.
\end{proposition}

The following lemma is well known; however, since we will use this equivalence in Section 3, we include its proof here.

\begin{lemma}\label{equiv}
If $\mathcal{G}$ is connected, fix an element $z\in\mathcal{G}_0$, then the categories $\mathcal{G}$ and $\mathcal{G}_z$ are equivalent.
\end{lemma}

\begin{proof}
Let $U:\mathcal{G}_z\to \mathcal{G}$  be the inclusion functor. Since $U$ is fully faithful and essentially surjective, $\mathcal{G}$ and $\mathcal{G}_z$ are equivalent categories.
\end{proof}

\begin{itemize}
\item The functor category $\set^\mathcal{G}$ is called the \textbf{category of $\mathcal{G}$-sets}. An object in $\set^\mathcal{G}$ is called a \textbf{$\mathcal{G}$-set}, and a morphism in $\set^\mathcal{G}$ is called a \textbf{$\mathcal{G}$-map}.
\item A functor $\mathcal{G}\to \mon$ is called a \textbf{$\mathcal{G}$-monoid}. Let $X$ be a $\mathcal{G}$-monoid, and let $U:\mon\to\set$ be the forgetful functor. Then the composite functor $U\circ X$ is a $\mathcal{G}$-set. We denote $U\circ X$ by $\overline{X}$.
\end{itemize}

\begin{definition}
\item If $X$ is a $\mathcal{G}$-set, then an action groupoid $\mathcal{G}\ltimes X$ is the category defined by the following conditions: 
\begin{enumerate}
\item $(\mathcal{G}\ltimes X)_0:=\coprod_{G\in \mathcal{G}_0} X(G),$
\item $(\mathcal{G}\ltimes X)(\langle G,x \rangle, \langle G',x' \rangle):=\{g\in\mathcal{G}(G,G')\mid X(g)(x)=x'\},$
\item The composition is induced by that in $\mathcal{G}$.
\end{enumerate}
A $\mathcal{G}$-set $X$ is called a \textbf{transitive} $\mathcal{G}$-set if $|\pi(\mathcal{G}\ltimes X)_0|=1$ holds.
\end{definition}

\begin{remark}\label{E}
The action groupoid $\mathcal{G}\ltimes X$ is also referred to as the category of elements of $X$ or the Grothendieck construction.
\end{remark}

\begin{definition}
If $\mathcal{G}$ is a groupoid, then \textbf{conjugation action} of $\mathcal{G}$
is the functor $\mathcal{G}^c:\mathcal{G}\to \mon$ defined by the following conditions: 
\begin{enumerate}
\item For an object $x\in \mathcal{G}_0$, define $\mathcal{G}^c(x):=\mathcal{G}_x.$
\item For a morphism $g:x\to y$, define $\mathcal{G}^c(g):\mathcal{G}_x\to \mathcal{G}_y$ by $x\mapsto gxg^{-1}$.
\end{enumerate}
\end{definition}

\begin{remark}
If $\mathcal{N}$ is a normal subgroupoid of $\mathcal{G}$, then $\mathcal{N}^c$ can be defined in the same way as $\mathcal{G}^c$ since $\mathcal{N}$ is closed under conjugation of $\mathcal{G}$.
\end{remark}

\subsection{The slice category of $\mathcal{G}$-sets}
In this subsection, we introduce the category of crossed $\mathcal{G}$-sets and prove its fundamental properties.

\begin{definition}
The Grothendieck ring $K_0(\set^{\mathcal{G}},\sqcup, \times)$ is called the \textbf{Burnside ring} of $\mathcal{G}$, and we denote it by $B(\mathcal{G})$. 
\end{definition}

\begin{definition}
If $S$ is a $\mathcal{G}$-set, then the slice category $\set^\mathcal{G}/S$ is called the category of \textbf{$\mathcal{G}$-set over $S$}. The product $*$ in the slice category $\set^\mathcal{G}/S$ is given by the fiber product: \[(X\to S)*(Y\to S)=(X\times_S Y\to S).\]
We refer to this product $*$ as the \textbf{Hadamard product}. Grothendieck ring \[K_0(\set^\mathcal{G}/X ,\sqcup,*)\] is called the \textbf{crossed Burnside ring of $X$ with Hadamard product}, and we denote it by $B_{\Had}(\mathcal{G},X)$.
\end{definition}

\begin{theorem}\label{AG}
Let $\mathcal{G}$ be a connected groupoid, let $y$ denote the Yoneda embedding $y:\mathcal{G}^{op}\to \set^\mathcal{G}$, and let $X$ be a $\mathcal{G}$-$\set$. Then the following statements hold:\\
\noindent
$(1)$ There exists a categorical isomorphism $(\mathcal{G}\ltimes X)^{op}\cong y\downarrow X$.\\
\noindent
$(2)$ There is an equivalence of categories $\set^{\mathcal{G}\ltimes X}\simeq \set^{\mathcal{G}}/X$.
\end{theorem}

\begin{proof}
$(1)$ By the general theory of category of elements, we have $(y\downarrow X)\cong \el(X)$. By Remark \ref{E}, we have $\el(X)\cong \mathcal{G}\ltimes X$. Moreover, define a functor $F:(\mathcal{G}\ltimes X)^{op}\to(\mathcal{G}\ltimes X)$ by the identity on objects and $g\mapsto g^{-1}$ on morphisms. Then $F$ is an isomorphism of categories. Combining these, we obtain $y\downarrow X\cong (\mathcal{G}\ltimes X)^{op}$.\\
\noindent
 $(2)$ By general theory of presheaf, we have $\set^{{(y\downarrow X)}^{op}}\simeq \set^\mathcal{G}/X$. By combining the equivalence of categories $\set^{{(y\downarrow X)}^{op}}\simeq \set^\mathcal{G}/X$ with $(1)$, we obtain $\set^{\mathcal{G}\ltimes X}\simeq \set^{\mathcal{G}}/X$.
\end{proof}

\begin{corollary}
The Burnside ring of $\mathcal{G}\ltimes X$ is isomorphic to the crossed Burnside ring of the $\mathcal{G}$-set $X$ endowed with the Hadamard product.
\end{corollary}

\begin{proof}
By Theorem \ref{AG}, there is an equivalence of categories $\set^{\mathcal{G}\ltimes X}\simeq \set^{\mathcal{G}}/X$. Taking Grothendieck rings on both sides, we obtain a ring isomorphism $B(\mathcal{G}\ltimes X)\cong B_{\mathrm{Had}}(\mathcal{G},X)$. 
\end{proof}

\section{The crossed Burnside rings of finite groupoids.}

In this section, we introduce a monoidal structure on the category of crossed $\mathcal{G}$-sets and investigate its properties, including the main theorems concerning the crossed Burnside ring of a groupoid.

\begin{theorem}\label{monoidal}
If $S$ is a $\mathcal{G}$-monoid, then the category $\set^\mathcal{G}/ \overline{S}$ is a monoidal category.
\end{theorem}

\begin{proof}
We now construct a monoidal structure on $\set^\mathcal{G}/ \overline{S}$. To this end, we specify the tensor product functor, the unit object, and the coherence isomorphisms. 

\noindent
$(1)$ Tensor product functor

The functor $\otimes:\set^\mathcal{G}/ \overline{S}\times\set^\mathcal{G}/ \overline{S}:\to \set^\mathcal{G}/ \overline{S}$ is defined by the following conditions: 
\begin{enumerate}
\item For an object $(\theta:X\to \overline{S},\tau:Y\to \overline{S})$, define the \[\theta\otimes\tau:X\times Y\to \overline{S}\] whose component at $G\in\mathcal{G}_0$ is the map \[(\theta\otimes\tau)_G:X(G)\times Y(G)\to \overline{S}(G)\]
by $(x,y)\mapsto \theta(x)\tau(y)$. \\
\item For $\mathcal{G}$-maps $\eta:X\to X'$ and $\eta:Y\to Y'$, we define $\otimes(\eta,\tau)$ by specifying its $G$-component as $\otimes(\eta,\tau)_G=\eta_G\times\tau_G$ for each $G\in\mathcal{G}_0$.
\end{enumerate}

\noindent
$(2)$ Unit object

Define the natural transformation $I:1\to \overline{S}$ whose component at $G\in\mathcal{G}_0$ is the map $I_G:1(G)\to \overline{S}(G)$ by $\bullet\mapsto 1_{S(G)}$. The unit object is denoted $I$.\\

Hencforth, we write $X$ for the $\mathcal{G}$-map $X\to\mathcal{G}^c$ when no confusion arises.\\

\noindent
$(3)$ Associator

For any $X,Y,Z\in\set^\mathcal{G}/ \overline{S}$, define the natural transformation \[\alpha_{X,Y,Z}:(X\otimes Y)\otimes Z\to X\otimes (Y\otimes Z)\]whose component at $G\in\mathcal{G}_0$ is the map \[(\alpha_{X,Y,Z})_G:(X(G)\otimes Y(G))\otimes Z(G)\to X(G)\otimes (Y(G)\times Z(G))\] by $((x,y),z)\mapsto (x,(y,z))$.\\

\noindent
$(4)$ Left unitor

For any $X\in\set^\mathcal{G}/ \overline{S}$, define the natural transformation \[l_X:I\otimes X\to X\] whose component at $G\in\mathcal{G}_0$ is the map \[(l_X)_G:I(G)\otimes X(G)\to X(G)\] by $(1_{S(G)},x)\mapsto x$. \\

\noindent
$(5)$ Right unitor

For any $X\in\set^\mathcal{G}/ \overline{S}$, define the natural transformation \[r_X:X\otimes I\to X\] whose component at $G\in\mathcal{G}_0$ is the map \[(r_X)_G:X(G)\otimes I(G)\to X(G)\] by $(x,1_{S(G)})\mapsto x$. \\

\noindent
$(6)$ Coherence conditions
\begin{itemize}
\item Pentagon axiom

For any $X,Y,Z,W\in\set^\mathcal{G}/ \overline{S}$, the diagram
\[\begin{tikzpicture}[>=Stealth,auto,scale=1,every node/.style={scale=0.92}]
\node (A) at (90:3) {$((W\otimes X)\otimes Y)\otimes Z$};
\node (B) at (18:3) {$(W\otimes X)\otimes (Y\otimes Z)$};
\node (C) at (-40:3.7) {$W\otimes (X\otimes (Y\otimes Z))$};
\node (D) at (-140:3.8) {$W\otimes ((X\otimes Y)\otimes Z)$};
\node (E) at (162:3) {$(W\otimes (X\otimes Y))\otimes Z$};
\draw[<-](B) -- node[right] {$\alpha_{W\otimes X,Y,Z}$}(A);
\draw[<-](C) -- node[right] {$\alpha_{W,X,Y\otimes Z}$}(B);
\draw[<-](C) -- node {$\id_W\otimes\alpha_{X,Y,Z}$}(D);
\draw[<-](D) -- node {$\alpha_{W,X\otimes Y,Z}$}(E);
\draw[<-](E) -- node {$\alpha_{W,X,Y}\otimes\id_Z$}(A);
\end{tikzpicture}\]
commutes.

\item Triangle axiom

For any $X,Y\in\set^\mathcal{G}/ \overline{S}$, the diagram
\[\xymatrix{
(X\otimes I)\otimes Y \ar[dr]_-{r_X\otimes\id_Y}\ar[r]^-{\alpha_{X,I,Y}}&X\otimes (I\otimes Y)\ar[d]^-{\id_X\otimes l_Y}\\
{}&X\otimes Y}
\]
commutes.

\end{itemize}
\end{proof}

\begin{lemma}\label{d}
If $X,Y,Z\in\set^\mathcal{G}/\overline{W}$, then 
\[X\otimes (Y\sqcup Z)\cong (X\otimes Y)\sqcup(X\otimes Z).\] 
\end{lemma}

\begin{proof}
This can be verified directly.
\end{proof}

\begin{definition}
The multiplication on $K_0(\set^\mathcal{G}/\overline{S},\sqcup)$ are induced by the monoidal structure on $\set^\mathcal{G}/ \overline{S}$
\[(Y\stackrel{\theta}{\to} \overline{S})\otimes (Z\stackrel{\tau}{\to}\overline{S}):=(Y\times Z\stackrel{\theta\times\tau}{\to} \overline{S}\times \overline{S}\stackrel{m}{\to} \overline{S})  \]
induced by the monoid multiplication. This ring is denoted by $B^c(\mathcal{G},S)$. $B^c(\mathcal{G},S)$ is called the \textbf{crossed Burnside ring of the $\mathcal{G}$-monoid S}.
In particular, $B^c(\mathcal{G},\mathcal{G}^c)$ is called the \textbf{crossed Burnside ring of $\mathcal{G}$}.
\end{definition}

\begin{theorem}
The category $\set^\mathcal{G}/ \mathcal{G}^c$ is a braided monoidal category. 
\end{theorem}

\begin{proof}
Let $\theta:X\to\mathcal{G}^c$ and $\tau:Y\to\mathcal{G}^c$ be $\mathcal{G}$-maps. 
For any $G\in\mathcal{G}_0$, we define the map $\eta_G:X(G)\times Y(G)\to Y(G)\times X(G)$ by $(x,y)\mapsto (Y(\theta_G(x))(y),x)$. For any $g:G\to G'$, the diagram
\[\xymatrix{
X(G)\otimes Y(G)\ar[d]_-{\eta_G}\ar[rr]^-{(X\otimes Y)(g)}&&X(G')\times Y(G')\ar[d]^-{\eta_G'}\\
Y(G)\otimes X(G)\ar[rr]_-{(Y\otimes X)(g)}&&Y(G')\times X(G')}\]
defined by
\[\begin{tikzpicture}[auto,|->]
\node(a) at (0,0) {$(Y(\theta_G(x))(y) ,x)$};
\node(b) at (5.5,0) {$(Y(g\circ\theta_G (x))(y), X(g)(x))$};
\node(c) at (0,1.5) {$(x,y)$};
\node(d) at (5.5,1.5) {$(X(g)(x), Y(g)(y))$};
\node(e) at (5.5,0.5) {$(Y(\theta_{G'} (X(g)(x)))\circ g)(y),X(g)(x))$};
\draw(d)--node {}(e);
\draw(a)--node[swap] {}(b);
\draw(c)--node[swap] {}(a);
\draw(c)--node {}(d);
\end{tikzpicture}\]
is commutative.
Therefore, $\eta:=\{\eta_G\}_{G\in\mathcal{G}_0}$ is a natural transformation from $X\otimes Y$ to $ Y\otimes X$. For each $G\in\mathcal{G}_0$, we define a map $\tau_G$ by $(x,y)\mapsto (y,Y(\theta^{-1}_G(y))(x))$. 
For any $G\in\mathcal{G}_0$, we obtain 

\[\begin{aligned} \eta_G\circ\tau_G(y,x) &=\eta_G(x,Y(\theta_G^{-1}(x))(y))\\&=((Y(\theta_G(x))\circ Y(\theta_G^{-1}(x)))(y),x)\\&=(y,x)\end{aligned}\] 
and 
\[\begin{aligned}\tau_G\circ\eta_G(y,x)&=\tau_G(Y(\theta_G(x))(y),x)\\&=(x,(Y(\theta^{-1}_G(x))\circ Y(\theta_G(x)))(y))\\&=(x,y)
.\end{aligned}\]
Since each component of $\eta$ is an isomorphism, the functors $X\otimes Y$ and $X\otimes Y$ are isomorphic.
By Theorem \ref{monoidal}, $\set^\mathcal{G}/ \mathcal{G}^c$ is a monoidal category. The following conditions $(1), (2)$ and $(3)$ hold for the monoidal structure of $\set^\mathcal{G}/ \mathcal{G}^c$ and $\eta$.

\noindent
$(1)$ For any $X,Y,Z\in \set^\mathcal{G}/ \mathcal{G}^c$, the diagram \[
\begin{tikzcd}[column sep=huge, row sep=huge]
(X\otimes Y)\otimes Z\arrow[r,"\alpha_{X,Y,Z}"]\arrow[d,"\eta_{X,Y}\otimes\id_z"']&X\otimes (Y\otimes Z)\arrow[r,"\eta_{X,Y\otimes Z}"]&
(Y\otimes Z)\otimes X\arrow[d,"\alpha_{Y,Z,X}"]\\
(Y\otimes X)\otimes Z\arrow[r,"\alpha_{Y,X,Z}"']&Y\otimes (X\otimes Z)\arrow[r,"\id_Y\otimes \eta_{X,Z}"']&Y\otimes (Z\otimes X)
\end{tikzcd}
\]
commutes. 

\noindent
$(2)$ For any $X,Y,Z\in \set^\mathcal{G}/ \mathcal{G}^c$, the diagram \[
\begin{tikzcd}[column sep=huge, row sep=huge]
X\otimes (Y\otimes Z)\arrow[r,"\alpha_{X,Y,Z}^{-1}"]\arrow[d,"\id_X\otimes\eta_{Y,Z}"']&(X\otimes Y)\otimes Z\arrow[r,"\eta_{X\otimes Y, Z}"]&
Z\otimes (X\otimes Y)\arrow[d,"\alpha^{-1}_{Z,X,Y}"]\\
X\otimes (Z\otimes Y)\arrow[r,"\alpha_{X,Z,Y}^{-1}"']&(X\otimes Z)\otimes Y\arrow[r," \eta_{X,Z}\otimes \id_Y"']&(Z\otimes X)\otimes Y
\end{tikzcd}
\]
commutes.  \\
Therefore, The category $\set^\mathcal{G}/ \mathcal{G}^c$ is a braided monoidal category. 
\end{proof}

\begin{theorem}\label{ME}
If $\mathcal{G}$ is connected and $z\in\mathcal{G}_0$, then the categories $\set^\mathcal{G}/\mathcal{G}^c$ and $\set^{\mathcal{G}_z}/\mathcal{G}^c_z$ are monoidally equivalent.
\end{theorem}

\begin{proof}
The proof relies on the correspondence established in Proposition \ref{equiv}. Let $R$ be a quasi-inverse of the functor $U$ constructed in Proposition \ref{equiv}. Then the equivalence of categories $\set^\mathcal{G}\simeq \set^{\mathcal{G}_z}$ is given by the following correspondence. 
\[\begin{array}{rccc}
\hat{R}:&\set^\mathcal{G}&\longrightarrow &
\set^{\mathcal{G}_z}\\
 &X&\longmapsto&XR,
\end{array}
\]
\[\begin{array}{rccc}
\hat{U}:&\set^{\mathcal{G}_z}&\longrightarrow &
\set^\mathcal{G}\\
 &X&\longmapsto&XU.
\end{array}
\]
Denote the isomorphisms $\hat{U} \circ \hat{R}\to \id_{\set^\mathcal{G}}$ by $\eta$ and $\hat{R} \circ \hat{U}\to \id_{\set^{\mathcal{G}^z}}$ by $\epsilon$.
Let $F:\set^{\mathcal{G}_z}/\mathcal{G}_z^c\to \set^\mathcal{G}/\mathcal{G}^c$ be the functor defined by the following conditions: 
\begin{enumerate}
\item For an object $\tau:Y\to\mathcal{G}_z^c$, define $F(\tau):=\eta_{\mathcal{G}^c}*\hat{U}(\tau)$.
\item For a morphism $f$, define $F(f):=\hat{U}(f)$.
\end{enumerate}

Define the functor $F:\set^\mathcal{G}/\mathcal{G}^c\to \set^{\mathcal{G}_z}/\mathcal{G}_z^c$ defined by the following conditions: 
\begin{enumerate}
\item For an object $\theta:X\to\mathcal{G}^c$, define $F(\theta):=\hat{R}(\theta)$.
\item For a morphism $f$, define $F(f):=\hat{R}(f)$.
\end{enumerate}

Let $\phi_0:=\id_I$, $\psi_0:=\id_{I'}$ and for each $X,Y\in \set^{\mathcal{G}}/\mathcal{G}^c$, $\phi_{X,Y}:=\id_{F(X)\otimes F(Y)}$, 
for each $Z,W\in \set^{\mathcal{G}_z}/\mathcal{G}_z^c$, $\psi_{Z,W}:=\id_{G(X)\otimes G(Y)}$.

We write $\alpha,r,l$ and $\alpha',r',l'$ to denote the coherent morphisms of $\set^{\mathcal{G}_z}/\mathcal{G}_z^c$, and $\set^\mathcal{G}/\mathcal{G}^c$, respectively.

The following conditions $(1)-(6)$ hold for the monoidal structure of $\set^\mathcal{G}/\mathcal{G}^c$ and $\set^{\mathcal{G}_x}/\mathcal{G}^c_x$, as well as for the functors $F$ and $G$.\\

\noindent
$(1)$ For any $X,Y,Z\in \set^\mathcal{G}/ \mathcal{G}^c$, the diagram 
\[
\begin{tikzcd}[column sep=small, row sep=small]
&(F(X)\otimes F(Y))\otimes F(Z)\arrow[dl,"\phi_{X,Y}\otimes\id_{F(Z)}"']\arrow[dr,"\alpha'_{F(X),F(Y),F(Z)}"]&\\
F(X\otimes Y)\otimes F(Z)\arrow[d,"\phi_{X\otimes Y,Z}"']&&F(X)\otimes (F(Y)\otimes F(Z))\arrow[d,"\id_{F(X)}\otimes\phi_{Y,Z}"]\\
F((X\otimes Y)\otimes Z)\arrow[dr,"F(\alpha_{X,Y,Z})"']&&F(X)\otimes F(Y\otimes Z)\arrow[dl,"\phi_{X,Y\otimes Z}"]\\
&F(X\otimes (Y\otimes Z))&
\end{tikzcd}
\]
\noindent commutes.\\

\noindent
$(2)$ For any $X\in \set^\mathcal{G}/ \mathcal{G}^c$, the diagram 
\[\xymatrix{
I'\otimes F(X)\ar[d]_-{\phi_0\otimes\id_{F(X)}}\ar[r]^-{l'_{F(X)}}&F(X)\\
F(I)\otimes F(X)\ar[r]_-{\phi_{I,X}}&F(I\otimes X)\ar[u]_-{F(l_I)
}}\]
\noindent commutes.\\

\noindent
$(3)$ For any $X\in \set^\mathcal{G}/ \mathcal{G}^c$, the diagram
\[\xymatrix{
F(X)\otimes I'\ar[d]_-{\id_{F(X)}\otimes\phi_0}\ar[r]^-{r'_{F(X)}}&F(X)\\
F(X)\otimes F(I)\ar[r]_-{\phi_{X,I}}&F(X\otimes I)\ar[u]_-{F(r_I)
}}\]
\noindent commutes.\\

\noindent
$(4)$ For any $X,Y,Z\in \set^{\mathcal{G}_z}/ \mathcal{G}_z^c$, the diagram
\[
\begin{tikzcd}[column sep=small, row sep=small]
&(G(X)\otimes G(Y))\otimes G(Z)\arrow[dl,"\psi_{X,Y}\otimes\id_{G(Z)}"']\arrow[dr,"\beta'_{G(X),G(Y),G(Z)}"]&\\
G(X\otimes Y)\otimes G(Z)\arrow[d,"\psi_{X\otimes Y,Z}"']&&G(X)\otimes (G(Y)\otimes G(Z))\arrow[d,"\id_{G(X)}\otimes\psi_{Y,Z}"]\\
G((X\otimes Y)\otimes Z)\arrow[dr,"G(\beta_{X,Y,Z})"']&&G(X)\otimes G(Y\otimes Z)\arrow[dl,"\psi_{X,Y\otimes Z}"]\\
&G(X\otimes (Y\otimes Z))&
\end{tikzcd}
\]
\noindent commutes.\\

\noindent
$(5)$ For any $X\in \set^{\mathcal{G}_z}/ \mathcal{G}_z^c$, the diagram
\[\xymatrix{
I\otimes G(X)\ar[d]_-{\psi_0\otimes\id_{G(X)}}\ar[r]^-{l'_{G(X)}}&G(X)\\
G(I)\otimes G(X)\ar[r]_-{\psi_{I',X}}&G(I\otimes X)\ar[u]_-{G(l_I)
}}\]
\noindent commutes.\\

\noindent
$(6)$ For any $X\in \set^{\mathcal{G}_z}/ \mathcal{G}_z^c$, the diagram
\[\xymatrix{
G(X)\otimes I\ar[d]_-{\id_{G(X)}\otimes\psi_0}\ar[r]^-{r'_{G(X)}}&G(X)\\
G(X)\otimes G(I)\ar[r]_-{\phi_{X,I'}}&G(X\otimes I')\ar[u]_-{G(r_I)
}}\]
\noindent commutes.\\

\noindent Therefore, $F$ and $G$ are monoidal functors.
For $\theta$ and $\eta$, the following conditions $(1)-(4)$ holds. 

\noindent
$(1)$ For any $X,Y\in \set^\mathcal{G}/ \mathcal{G}^c$, the diagram
\[\xymatrix{
(G\circ F)(X)\otimes (G\circ F)(Y)\ar[d]_-{\eta_X\otimes\eta_Y}\ar[rrr]^-{\psi_{F(X),F(Y)}\circ\,G(\phi_{X,Y})}&&&(G\circ F)(X\otimes Y)\ar[d]^-{\eta_{X\otimes Y}}\\
X\otimes Y\ar[rrr]_-{\id_{X\otimes Y}}&&&X\otimes Y
}\]
\noindent commutes.\\

\noindent
$(2)$ The diagram
\[\xymatrix{
I \ar[dr]_-{\id_I}\ar[r]^-{G(\phi_0)}&(G\circ F)(I)\ar[d]^-{\eta_I}\\
{}&I}\]
\noindent commutes.\\
\noindent
$(3)$ For any $X,Y\in \set^{\mathcal{G}_z}/ \mathcal{G}_z^c$, the diagram
\[\xymatrix{
(F\circ G)(X)\otimes (F\circ G)(Y)\ar[d]_-{\epsilon_X\otimes\epsilon_Y}\ar[rrr]^-{\phi_{G(X),G(Y)}\circ\, F(\phi_{X,Y})}&&&(F\circ G)(X\otimes Y)\ar[d]^-{\epsilon_{X\otimes Y}}\\
X\otimes Y\ar[rrr]_-{\id_{X\otimes Y}}&&&X\otimes Y
}\]
\noindent commutes.\\

\noindent
$(4)$ The diagram
\[\xymatrix{
I' \ar[dr]_-{\id_{I'}}\ar[r]^-{F(\psi_0)}&(F\circ G)(I')\ar[d]^-{\epsilon_{I'}}\\
{}&I'}\]
\noindent commutes.\\
Hence, $C$ and $D$ are monoidally equivalent.
\end{proof}

\begin{theorem}\label{Inj}
The Burnside ring of $\mathcal{G}$ can be embedded into the crossed Burnside ring of the $\mathcal{G}$-monoid $S$.  
\end{theorem}

\begin{proof}
Since $\set^\mathcal{G}$ has finite products, it becomes a monoidal category under finite product.  
Let $\alpha,r,l$ be coherent morphisms of $\set^\mathcal{G}$, and let $\alpha',r',l'$ be coherent morphisms of $\set^\mathcal{G}/S$.
Let $F:\set^\mathcal{G}\to \set^\mathcal{G}/S$ be the functor defined by the following conditions: 
\begin{enumerate}
\item For an object $X$, define $F(X)$ to be  the natural transformation whose component at $G$ is the map $F(X)_G:X(G)\to S(G)$ by $x_G\mapsto 1_{s_G}$ .
\item For a morphism $\theta$, define $F(\theta):=\theta$.
\end{enumerate}

Define the functor $G:\set^\mathcal{G}/S\to \set^\mathcal{G}$ defined by the following conditions: 
\begin{enumerate}
\item For an object $f:X\to S$, define $G(f):=X$.
\item For a morphism $\theta$, define $G(\theta):=\theta$.
\end{enumerate}

Let $\phi_0:=\id_I$, and for each $X,Y\in \set^{\mathcal{G}}$, let $\phi_{X,Y}=\id_{F(X)\otimes F(Y)}$.

The following conditions $(1)-(3)$ hold for the monoidal structure of $\set^\mathcal{G}$ and 
$\set^{\mathcal{G}}/X$, as well as for the functors $F$ and $G$.\\
\noindent
$(1)$ For any $X,Y,Z\in \set^\mathcal{G}$, the diagram
\[
\begin{tikzcd}[column sep=small, row sep=small]
&(F(X)\otimes F(Y))\otimes F(Z)\arrow[dl,"\phi_{X,Y}\otimes\id_{F(Z)}"']\arrow[dr,"\alpha'_{F(X),F(Y),F(Z)}"]&\\
F(X\otimes Y)\otimes F(Z)\arrow[d,"\phi_{X\otimes Y,Z}"']&&F(X)\otimes (F(Y)\otimes F(Z))\arrow[d,"\id_{F(X)}\otimes\phi_{Y,Z}"]\\
F((X\otimes Y)\otimes Z)\arrow[dr,"F(\alpha_{X,Y,Z})"']&&F(X)\otimes F(Y\otimes Z)\arrow[dl,"\phi_{X,Y\otimes Z}"]\\
&F(X\otimes (Y\otimes Z))&
\end{tikzcd}
\]
\noindent commutes.\\

\noindent
$(2)$ For any $X\in \set^\mathcal{G}$, the diagram
\[\xymatrix{
I'\otimes F(X)\ar[d]_-{\phi_0\otimes\id_{F(X)}}\ar[r]^-{l'_{F(X)}}&F(X)\\
F(I)\otimes F(X)\ar[r]_-{\phi_{I,X}}&F(I\otimes X)\ar[u]_-{F(l_I)
}}\]
\noindent commutes.\\

\noindent
$(3)$ For any $X\in \set^\mathcal{G}$, the diagram
\[\xymatrix{
F(X)\otimes I'\ar[d]_-{\id_{F(X)}\otimes\phi_0}\ar[r]^-{r'_{F(X)}}&F(X)\\
F(X)\otimes F(I)\ar[r]_-{\phi_{X,I}}&F(X\otimes I)\ar[u]_-{F(r_I)
}}\]
\noindent commutes.\\
\noindent
Therefore, $F$ and $G$ are monoidal functors.
For every an object $X$, we have $(G\circ F)_0(X)=X$. For a morphism $\theta$, we obtain $(G\circ F)_1(\theta)=\theta$.
Therefore, we obtain $G\circ F=\id_{\set^\mathcal{G}}$. Moreover, $F$ preserves coproducts of $\set^\mathcal{G}$.
The ring homomorphism induced by $F$ is injective, so the Burnside ring $B(\mathcal{G})$ can be embedded into the crossed Burnside ring $B^c(\mathcal{G})$. 
\end{proof}

\begin{remark}
Theorem \ref{Inj} corresponds to a generalization of the construction of $\epsilon_{!}$, $\eta_{!}$ in \cite{OY04}(2.10).
\end{remark}

\begin{theorem}\label{BT}
If $\mathcal{G}\cong\coprod^n_{i=1}\mathcal{G}_{[i]}$ is the decomposition of $\mathcal{G}$ into its connected components, then two rings $B^c(\mathcal{G})$ and $\prod_{x\in\pi_0(\mathcal{G})}B^c(\mathcal{G}_x)$ are isomorphic.
\end{theorem}

\begin{proof}
 By Theorem \ref{ME},We obtain 
\[\begin{aligned}\set^{\mathcal{G}}/\mathcal{G}^c&\cong\prod^n_{i=1} \set^{\mathcal{G}_{[i]}}/\mathcal{G}_{[i]}^c \\& \simeq \prod_{x\in\pi_0(\mathcal{G})} \set^{\mathcal{G}_x}/\mathcal{G}_x^c.
\end{aligned}\]
Taking the Grothendieck ring, we obtain $B^c(\mathcal{G})\cong\prod_{x\in\pi_0(\mathcal{G})} B^c(\mathcal{G}_x)$.
\end{proof}

\begin{remark}
The isomorphism $B^c(\mathcal{G})\cong\prod_{x\in\pi_0(\mathcal{G})}B(\mathcal{G}_x)$ is induced by a non-canonical equivalence of categories.
\end{remark}

\begin{remark}
Although we don't carry out explicit computations of crossed Burnside ring for groupoids here, Theorem \ref{BT} enables us to compute the crossed Burnside ring of a groupoid using methods developed for the crossed Burnside ring of a group. 
\end{remark}

\section*{Acknowledgement}
The author sincerely thanks his supervisor, Professor Fumihito Oda, for his support and valuable comments.


\begin{thebibliography}{99}
\bibitem[AMP20]{AMP20} J. \'{A}vila; V. Mar\'{i}n and H. Pinedo, Isomorphism theorems for groupoids and some applications, International Journal of Mathmatics and Mathematical Sciences, (2020), Art.ID 3967368, 10. MR 4073238.
\bibitem[BD20]{BD20} P. Balmer and I. Dell' Ambrogio, Mackey 2-functors and Mackey 2-motives (European Mathmatical 'Society (EMS), Z\"{u}rich, 2020).
\bibitem[Bou03]{Bou03} S. Bouc. The $p$-blocks of the Mackey algebra. Algebr. Represent. Theory, 6(5): 515-543, 2003.
\bibitem[Bra26]{Bra26} H. Brandt, \"{U}ber eine Verallgemeinerung des Gruppenbegriffes, Math. Ann., \textbf{96} (1926), 360-366.
\bibitem[ES23]{ES23} L. EI Kaoutit and L. Spinosa, "On Burnside Theory for groupoids", Bull. Math. Soc. Sci. Math. Roumanie (2023) Tome \textbf{66}(\textbf{114}), No.1, 41--87.
\bibitem[Iva02]{Iva02} G. Ivan, Algebraic constructions of Brandt Groupoids, Proceeding of the Algebra Symposium, Babes-Bolyai University Cluj, 69-90, 2002.
\bibitem[Wil69]{Wil69} N.H. Williams, On Grothendieck universes. Conpositio Mathematica, tome 21, No.1, (1969), 1-3. 
\bibitem[OY01]{OY01} F. Oda and T. Yoshida. Crossed Burnside rings. I. The fundamental theorem, J. Algebra, \textbf{236} (2001), no. 1, 29-79.
\bibitem[OY04]{OY04} F. Oda and T. Yoshida. Crossed Burnside rings. II. The Dress construction of a Green functor, J. Algebra, \textbf{282} (2004), no. 1, 58-82.
\bibitem[Rog13]{Rog13} B. Rognerud. Equivalences de blocs d'alg\`{e}bres de Mackey. PhD thesis, Universit\'{e} de Picardie Jules Verne, December 2013.
\bibitem[Yos97]{Yos97} T. Yoshida. Crossed $G$-sets and crossed Burnside rings.\\ S\={u}rikaisekikenk\={y}usho K\={o}ky\={u}roku, (991):1-15, 1997. Group theory and combinatorial mathematics (Japanese) (Kyoto, 1996).
\end{thebibliography}
\end{document}